\documentclass[10pt]{article}
\font\smallit=cmti10
\font\smalltt=cmtt10

\usepackage{color}
\usepackage{hyperref}
\usepackage{enumerate}
\usepackage{color}
\usepackage{breakurl}
\usepackage{comment}
\newcommand{\bburl}[1]{\textcolor{blue}{\url{#1}}}

\usepackage{amssymb,latexsym,amsmath,epsfig,amsthm} %% Add other packages as necessary

\makeatletter

\renewcommand\section{\@startsection {section}{1}{\z@}
{-30pt \@plus -1ex \@minus -.2ex}
{2.3ex \@plus.2ex}
{\normalfont\normalsize\bfseries\boldmath}}

\renewcommand\subsection{\@startsection{subsection}{2}{\z@}
{-3.25ex\@plus -1ex \@minus -.2ex}
{1.5ex \@plus .2ex}
{\normalfont\normalsize\bfseries\boldmath}}

\renewcommand{\@seccntformat}[1]{\csname the#1\endcsname. }

\makeatletter
\newcommand{\myitem}[1]{%
\item[#1]\protected@edef\@currentlabel{#1}%
}
\makeatother

\makeatother

\newtheorem{thm}{Theorem}[section]
\newtheorem{conj}[thm]{Conjecture}
\newtheorem{cor}[thm]{Corollary}
\newtheorem{lem}[thm]{Lemma}

\newtheorem{exa}[thm]{Example}
\newtheorem{defi}[thm]{Definition}

\newtheorem{rek}[thm]{Remark}

\newcommand{\Mod}[1]{\ \mathrm{mod}\ #1}
\begin{document}

\begin{center}
\uppercase{\bf Generalizations of a Curious Family of MSTD Sets Hidden by Interior Blocks}
\vskip 20pt
{\bf H\`ung Vi\d{\^e}t Chu}\\
{\smallit Department of Mathematics, Washington and Lee University, Lexington, VA 24450}\\
{\tt chuh19@mail.wlu.edu}\\
\vskip 10pt
{\bf Noah Luntzlara}\\
{\smallit Department of Mathematics, University of Michigan, Ann Arbor, MI 48109}\\
{\tt nluntzla@umich.edu}\\
\vskip 10pt
{\bf Steven J. Miller}\\
{\smallit Department of Mathematics and Statistics, Williams College, Williamstown, MA 01267}\\
{\tt sjm1@williams.edu}\\
\vskip 10pt
{\bf Lily Shao}\\
{\smallit Department of Mathematics and Statistics, Williams College, Williamstown, MA 01267}\\
{\tt ls12@williams.edu}\\
\end{center}
\vskip 20pt
\centerline{\smallit Received: , Revised: , Accepted: , Published: } % We will fill in the dates
\vskip 30pt

\centerline{\bf Abstract}
\noindent
A set $A$ is MSTD (more-sum-than-difference) or sum-dominant if $|A+A|>|A-A|$, and is RSD (restricted-sum dominant) if $|A\hat{+}A|>|A-A|$, where $A\hat{+}A$ is the set of sums of distinct elements in $A$. We study an interesting family of MSTD sets that have appeared many times in the literature (see the works of Hegarty, Martin and O'Bryant, and Penman and Wells). While these sets seem at first glance to be ad hoc, looking at them in the right way reveals a nice common structure. In particular, instead of viewing them as explicitly written sets, we write them in terms of differences between two consecutive numbers in increasing order. We denote this family by $\mathcal{F}$ and investigate many of its properties. Using $\mathcal{F}$, we are able to generate many sets $A$ with high value of $\log|A+A|/\log|A-A|$, construct sets $A$ with a fixed $|A+A|-|A-A|$ more economically than previous authors, and improve the lower bound on the proportion of RSD subsets of $\{0,1,2,\dots,n-1\}$ to about $10^{-25}$ (the previous best bound was $10^{-37}$). Lastly, by exhaustive computer search, we find six RSD sets with cardinality $15$, which is one lower than the smallest cardinality found to date, and find that $30$ is the smallest diameter of RSD sets.

\pagestyle{myheadings}
\markright{\smalltt INTEGERS: 19 (2019)\hfill}
\thispagestyle{empty}
\baselineskip=12.875pt
\vskip 30pt

\section{Introduction}

%%%%%%%%%%%%%%%%%%%%%%%%%%%%%%%%%%%%%%%%%%%%%%%%%%%%%%%%%%%%%%%%%%
%%%%%%%%%%%%%%%%%%%%%%%%%%%%%%%%%%%%%%%%%%%%%%%%%%%%%%%%%%%%%%%%%%
%%%%%%%%%%%%%%%%%%%%%%%%%%%%%%%%%%%%%%%%%%%%%%%%%%%%%%%%%%%%%%%%%%
\subsection{Background}
Given a finite set of
non-negative integers $A$, the \textit{sum set} is defined to be
\[
A+A := \{a_i+a_j: a_i,a_j\in
A\}
\]
and the \textit{difference set} to be
\[
A-A := \{a_i-a_j:a_i,a_j\in A\};
\]
$A$ is said to be
\textit{sum-dominated} or \textit{MSTD} (more sums than differences) if $|A+A|>|A-A|$, \textit{balanced} if $|A+A|=|A-A|$, and
\textit{difference-dominated} if $|A+A|<|A-A|$. Also, we define the \textit{restricted sum set} to be
\[
A\hat{+}A: =\{a_i+a_j:a_i, a_j\in A\mbox{ and }a_i\neq a_j\}.
\]
We call a set $A$ \textit{restricted sum-dominant (RSD)} if $|A\hat{+}A|>|A-A|$. We could similarly define a restricted difference set by only considering differences of distinct elements, but this would amount to removing the number $0$ from the difference set, decreasing the cardinality of $A-A$ by one and not substantially changing the questions about RSD sets. Thus we avoid this definition.

Since Conway gave an early example of an MSTD set in 1969\footnote{see footnote 1 of \cite{Na2}}, research on MSTDs has made incredible progress; see \cite{He,Ma,Na1,Na2,Ru1,Ru2,Ru3} for some of the earlier results and constructions. One of the most notable papers is by Martin and O'Bryant \cite{MO}. They proved the proportion of MSTD subsets of $\{0,1,2,\dots,n-1\}$ is bounded below by a positive constant as $n\rightarrow \infty$. However, the proof is probabilistic and does not give explicit constructions of MSTD sets. Later, Miller, Orosz and Scheinerman \cite{MOS} gave an explicit construction of a dense family of MSTD sets (previous bounds were exponentially small). They showed that as $n\rightarrow \infty$, the proportion of MSTD subsets of $\{0,1,2,\dots,n-1\}$ that are in their family is at least $C/n^4$ for some constant $C$.\footnote{With a more refined analysis, the density can be improved to $C/n^2$.} The current record of a dense family belongs to Zhao \cite{Zh1} with a family of density $C/n$.

In this paper, we focus on a particular family of MSTD sets (which we denote by $\mathcal{F}$) that has appeared many times in the literature. These sets appear to arise at random and have no particular order, but if we look at them in the right way, they are very well-structured. In addition, our family $\mathcal{F}$ has many nice properties that we will explore, despite not being dense.

We first provide some examples of sets in $\mathcal{F}$ that have been discussed in the literature. The following sets are found in \cite{MO}:
\begin{align*}
    S_2\ &=\ \{0, 1, 2, 4, 5, 9, 12, 13, 14, 16, 17\},\\
    S_4\ &=\ \{0, 1, 2, 4, 5, 9, 12, 13, 17, 20, 21, 22, 24, 25\}.
\end{align*}
And these sets below are found in \cite{He}:
\begin{align*}
    A_4\ &=\ \{0, 1, 2, 4, 5, 9, 12, 13, 14\},\\
    A_{12}\ &=\ S_2, \\ %\mbox{ } X\ =\ S_4,\\
    A_{15}\ &= \ \{0, 1, 2, 4, 5, 9, 12, 13, 17, 20, 21, 22, 24, 25, 29, 32, 33, 37, 40, 41, 42, 44, 45\}.
\end{align*}
Last but not least, the following sets are found in \cite{PW}:
\begin{align*}
    T'_j\ &= \ \{0,2\}\cup\{1,9,\ldots,1+8j\}\cup\{4,12,\ldots,4+8j\}\\
    &\cup\{5,13,\ldots,5+8j\}\cup\{6+8j, 8(j+1)\}\mbox{ (Theorem 1)},\\
    T_j\ &= \ \{0,2\}\cup\{1,9,\ldots,1+8(j+1)\}\cup\{4,12,\ldots,4+8j\}\\
    &\cup\{5,13,\ldots,5+8j\}\cup\{6+8j, 8(j+1)\}\mbox{ (Theorem 4)},\\
    R_j\ &= \ \{1,4\}\cup\{0,12,\ldots,12j\}\cup\{2,14,\ldots, 2+12j\}\\
    &\cup\{7,19,\ldots,7+12j\}\cup\{8,20,\ldots,8+12j\}\cup\{3+12j,6+12j\}\mbox{ (Theorem 6)}.
\end{align*}
These sets play important roles in the papers which initially described them. For example, $A_{12}$ is used to prove Theorem 8 in \cite{He}, which states that there exists a positive constant lower bound for the proportion of sets with fixed cardinalities of sum sets and difference sets. The sets $T_j$ and $T'_j$ give explicit construction of RSD sets. The set $R_j$ gives a set $A$ with the highest known value of $\log|A+A|/\log |A-A|$. The study of our family $\mathcal{F}$ was motivated by trying to describe a common pattern among these remarkable MSTD sets. Although they arose in somewhat different situations, and were presented ad hoc for the purposes of each of the papers which described them, it turned out that all these sets have something in common, and all belong to $\mathcal{F}$.

The members of the family $\mathcal{F}$ have many nice properties, including (1) sets $A$ with large values of $\log |A+A|/\log |A-A|$\footnote{The highest value of $\log|A+A|/\log|A-A|$ for a set $A$ in $\mathcal{F}$ is about 1.03059.}; (2) economical construction of sets $A$ with fixed $|A+A|-|A-A|$; (3) demonstration of Spohn's conjecture (1973); (4) compactness; (5) more constructions of RSD subsets; and (5) examples of small fringes.

%%%%%%%%%%%%%%%%%%%%%%%%%%%%%%%%%%%%%%%%%%%%%%%%%%%%%%%%%%%%%%%%%%
%%%%%%%%%%%%%%%%%%%%%%%%%%%%%%%%%%%%%%%%%%%%%%%%%%%%%%%%%%%%%%%%%%
%%%%%%%%%%%%%%%%%%%%%%%%%%%%%%%%%%%%%%%%%%%%%%%%%%%%%%%%%%%%%%%%%%
\subsection{Notation and Main Results}\label{not}
Let nonnegative numbers $a\le b$ be chosen. Let $[a,b]=\{x\,|\,a\le x\le b\}$ and $[a,b]_q=\{x\,|\,x\equiv a\pmod{q}\mbox{ and }a\le x\le b\}$.

We use a different notation to represent sets of integers; it was first introduced by Spohn \cite{Sp} (1973): Given a set $S=\{a_1,a_2,\ldots,a_n\}$, we arrange its elements in increasing order and form the sequence of differences between consecutive elements. Suppose that $a_1<a_2<\cdots<a_n$, then our sequence is $a_2-a_1,a_3-a_2,a_4-a_3,\ldots,a_n-a_{n-1}$ and we represent $$S\ =\ (a_1\,|\,a_2-a_1,a_3-a_2,a_4-a_3,\ldots,a_n-a_{n-1}).$$

For example, if $S=\{3,2,5,10,9\}$, we would arrange the elements in increasing order to get $2<3<5<9<10$, then write $S=(2\,|\,1,2,4,1)$.

We call the sequence $a_2-a_1, a_3-a_2, a_4-a_3,\dots,a_n-a_{n-1}$ the \textit{sequence of consecutive differences (SCD)}. The advantage of this notation is that differences between elements of $S$ correspond to sums of consecutive runs in the SCD. For example, look at the SCD $1,2,4,1$. We know that $7$ is in the difference set because the run $1,2,4$ sums up to $7$.

We now define the \textit{$\mathcal{F}$ family} and \textit{interior blocks}.

\begin{defi}
Let $M^k$ denote the sequence $1,\underbrace{4,\ldots,4}_{k\text{-times}},3$. We define $\mathcal{F}$ to be the family of sets with SCD
$$1,1,2,1,M^{k_1},M^{k_2},\ldots,M^{k_\ell},M_1$$
where $\ell, k_1,\ldots, k_\ell$ are positive integers, and $M_1$ is either $1,1$ or $1,1,2$ or $1,1,2,1$.
\end{defi}

\begin{rek}
It can be verified that all the sets $S_2, S_4, A_4, A_{12}, A_{15}, T_j', T_j, R_j$ are in $\mathcal{F}$.
\end{rek}

\begin{conj}\label{conjec}
All sets in $\mathcal{F}$ are MSTD.
\end{conj}

\begin{exa} The set
\begin{align*}S\ &=\ (0\,|\,1,1,2,1,4,3,1,4,4,3,1,4,4,4,3,1,1,2,1)\\&= \ \{0,1,2,4,5,9,12,13,17,21,24,25,29,33,37,40,41,42,44,45\}\end{align*} has $|S+S|-|S-S|=86-83=3$.
\end{exa}

In this paper, we prove that the conjecture is true for a periodic subfamily of $\mathcal{F}$.

\begin{defi}\label{interdef}[Interior block]
Consider a set $S$ with its SCD. Let $B$ denote a consecutive subsequence of the SCD. Suppose there exists $N\in\mathbb{N}$ such that for $k\ge N$, the sets $S_k$ with SCD constructed by repeating $B$ for $k$ times in place of $B$ are MSTD sets. Then we call $B$ an interior block and let $|B|$ denote the length of the interior block.
\end{defi}

A natural question to ask is what are the possible values of $|B|$, which is addressed by the following theorem:

\begin{thm}\label{periodic}
Let $k$ and $\ell$ be arbitrary positive integers. The following three subfamilies of $\mathcal{F}$ consist of MSTD sets; we specify the exact values of $|S+S|-|S-S|$.
\begin{enumerate}
\item \label{firsti}
The family $\{S_{k,\ell}\,|\,k,\ell\in \mathbb{N}\}$, where \begin{align*}
    S_{k,\ell}=(0\,|\,1,1,2,
    \underbrace{1,{\underbrace{4,\ldots,4}_{k\text{-times}},3},1,{\underbrace{4,\ldots,4}_{k\text{-times}},3},\ldots{},
    {1,\underbrace{4,\ldots,4}_{k\text{-times}}},3}_{\ell\text{-times}},1,1,2,1)
\end{align*}
has $|S_{k,\ell}+S_{k,\ell}|-|S_{k,\ell}-S_{k,\ell}|=2\ell$.
\item \label{secondi}
The family $\{S'_{k,\ell}\,|\,k,\ell\in \mathbb{N}\}$, where\begin{align*}
    S'_{k,\ell}=(0\,|\,1,1,2,
    \underbrace{1,{\underbrace{4,\ldots,4}_{k\text{-times}},3},1,{\underbrace{4,\ldots,4}_{k\text{-times}},3},\ldots{},
    {1,\underbrace{4,\ldots,4}_{k\text{-times}}},3}_{\ell\text{-times}},1,1,2)
\end{align*}
has $|S'_{k,\ell}+S'_{k,\ell}|-|S'_{k,\ell}-S'_{k,\ell}|=2\ell-1$.
\item \label{thirdi}
The family $\{S''_{k,\ell}\,|\,k,\ell\in \mathbb{N}\}$, where \begin{align*}
    S''_{k,\ell}=(0\,|\,1,1,2,
    \underbrace{1,{\underbrace{4,\ldots,4}_{k\text{-times}},3},1,{\underbrace{4,\ldots,4}_{k\text{-times}},3},\ldots{},
    {1,\underbrace{4,\ldots,4}_{k\text{-times}}},3}_{\ell\text{-times}},1,1)
\end{align*}
has $|S''_{k,\ell}+S''_{k,\ell}|-|S''_{k,\ell}-S''_{k,\ell}|=\ell$.
\end{enumerate}
We call $\{S_{k,\ell}|k,\ell\in\mathbb{N}\}\cup \{S'_{k,\ell}|k,\ell\in\mathbb{N}\}\cup \{S''_{k,\ell}|k,\ell\in\mathbb{N}\}$ the $\mathcal{F}_{per}$ family, which is a periodic subfamily of the larger family $\mathcal{F}$.
\end{thm}
\begin{exa}
\begin{enumerate}
    \item The set \begin{align*} S_{2,3}\ &=\ (0\,|\,1,1,2,1,4,4,3,1,4,4,3,1,4,4,3,1,1,2,1)\\&=\ \{0,1,2,4,5,9,13,16,17,21,25,28,29,33,37,40,41,42,44,45\}\end{align*}
    has $|S_{2,3}+S_{2,3}|-|S_{2,3}-S_{2,3}|=85-79=6=2\cdot 3$.
    \item The set \begin{align*} S'_{3,2}\ &=\ (0\,|\,1,1,2,1,4,4,4,3,1,4,4,4,3,1,1,2)\\&=\
    \{0,1,2,4,5,9,13,17,20,21,25,29,33,36,37,38,40\} \end{align*}
    has $|S'_{3,2}+S'_{3,2}|-|S'_{3,2}-S'_{3,2}|=72-69=3=2\cdot 2-1$.
    \item The set \begin{align*}S''_{3,2}\ &=\ (0\,|\,1,1,2,1,4,4,4,3,1,4,4,4,3,1,1)\\ &=\
    \{0,1,2,4,5,9,13,17,20,21,25,29,33,36,37,38\}\end{align*}
    has $|S''_{3,2}+S''_{3,2}|-|S''_{3,2}-S''_{3,2}|=67-65=2$.
\end{enumerate}
\end{exa}
\begin{rek}
Theorem \ref{periodic} answers a question raised by Spohn \cite{Sp}, on whether the interior block must contain at least 3 elements. The answer is no: set $\ell=1$ and choose $k\ge 2$ for sets in Theorem \ref{periodic}. Then, we can use either the interior block $4$ or $4,4$. In other words, we have shown that for all $i\in \mathbb{N}$, there exists an interior block $B$ with $|B|=i$. Also, this theorem provides an infinite family of sets which demonstrate Conjecture 6 in \cite{Sp}.\footnote{The repetition of certain interior blocks can cause the number of sums to be increased by a greater constant than that by which the number of differences is increased.}
\end{rek}

There has been a lot of interest in finding sets $A$ with large values of the ratio $\log |A+A|/\log |A-A|$. An early high ratio of about $1.0208$ was given by Hegarty \cite{He}, and later, a higher ratio of about $1.02313$ was found by Asada et al.\ \cite{AMMS}. The current highest is about $1.03059$ found by Penman and Wells in \cite{PW}, which is much higher than previous results. We observe that both examples of sets $A$ with high ratios from \cite{He} and \cite{PW} belong to $\mathcal{F}$. We offer examples of several sets in $\mathcal{F}$ that give higher ratios than the ones given in \cite{AMMS} and  \cite{He}; there are at least 22 sets $A$ in $\mathcal{F}$ with $\log |A+A|/\log |A-A|>1.3$.

Furthermore, the family $\mathcal{F}$ gives an economical way (in the sense of having a relatively small width between its minimum and maximum element) to construct a set $A$ with any specific value of $|A+A|-|A-A|$. Martin and O'Bryant proved that for a given $x\in\mathbb{N}$, there exists a set $A\subseteq \left[0,17|x|\right]$ such that $|A+A|-|A-A|=x$, which is significantly more efficient than the base expansion method.\footnote{We can generate an infinite family of MSTD sets from a given MSTD set through the base expansion method. Let $A$ be an MSTD set, and let $A_{k,m}=\{\sum_{i=1}^{k}a_im^{i-1}:a_i\in A\}$. If $m$ is sufficiently large, then $|A_{k,m}\pm A_{k,m}| = |A\pm
A|^k$ and $|A_{k,m}|=|A|^k$.} With subfamilies of $\mathcal{F}$, we further improve this.

\begin{thm}\label{econ}
Given $x\in\mathbb{N}$, there exists a set $A\subseteq [0,12+4x]$ such that $|A+A|-|A-A|=x$. Furthermore, it is impossible to construct $A\subseteq [0,f(x)]$ such that $|A+A|-|A-A|=x$, where $f(x)$ is sub-linear. This means that a linear growth of the interval containing $A$ is the best we can do.
\end{thm}

Finally, we improve the lower bound for the proportion of RSD subsets of $[0,n-1]$ as $n$ goes to infinity. RSD implies MSTD, and compared to MSTD sets, RSD sets are much less common: exhaustive computer search shows that there are no RSD subsets of $[0,29]$, while there are at least $4.5\cdot 10^{5}$ MSTD sets in the same interval. In \cite{PW}, the lower bound on the proportion of RSD subsets of $[0,n-1]$ as $n\to \infty$ is about $10^{-37}$; we improve this bound to $4.135\cdot 10^{-25}$ by a better fringe formed by using $\mathcal{F}$.\footnote{There are exactly 6 RSD subsets of $[0,30]$ and 16 RSD subsets of $[0,31]$. Based on these observations, we predict the true proportion of RSD subsets of $[0,n-1]$ as $n\rightarrow\infty$ to be about $3\cdot 10^{-9}$.}

\begin{thm}\label{boundsforRSD}
For $n\ge 81$, the proportion of RSD subsets of $[0,n-1]$ is at least $4.135\cdot 10^{-25}$.
\end{thm}

\ \\

This work was supported by NSF Grants DMS1561945 and DMS1659037, the University of Michigan, Washington and Lee University, and Williams College. We thank the referee for helpful comments on an earlier draft.

\ \\

%%%%%%%%%%%%%%%%%%%%%%%%%%%%%%%%%%%%%%%%%%%%%%%%%%%%%%%%%%%%%%%%%%%%%%%%%%%%%%%%%%%%%%%%%%%%%%%%%%%%%%%%%%%%%%%%%%%%%%%%%%%%%%%%%%%%%%%
%%%%%%%%%%%%%%%%%%%%%%%%%%%%%%%%%%%%%%%%%%%%%%%%%%%%%%%%%%%%%%%%%%%%%%%%%%%%%%%%%%%%%%%%%%%%%%%%%%%%%%%%%%%%%%%%%%%%%%%%%%%%%%%%%%%%%%%
%%%%%%%%%%%%%%%%%%%%%%%%%%%%%%%%%%%%%%%%%%%%%%%%%%%%%%%%%%%%%%%%%%%%%%%%%%%%%%%%%%%%%%%%%%%%%%%%%%%%%%%%%%%%%%%%%%%%%%%%%%%%%%%%%%%%%%%

\section{Proof of Theorem \ref{periodic}}
We only prove (\ref{firsti}) for conciseness, since the proofs of (\ref{secondi}) and (\ref{thirdi}) are similar. For (\ref{firsti}), the case when $k=1$ is Theorem 4 in \cite{PW},
so we prove the case $k\ge 2$.

\begin{lem}\label{forms}
Fix $i\ge 0$ and consider $6+4i$. For some fixed $k$ and $\ell$, if there is a run that sums up to $6+4i$ in the SCD
\[1,1,2,
    \underbrace{1,{\underbrace{4,\ldots,4}_{k\text{-times}},3},\ldots{},
    {1,\underbrace{4,\ldots,4}_{k\text{-times}}},3}_{\ell\text{-times}},1,1,2,1,
    \]
then the run is one of the forms
\begin{enumerate}
\myitem{(A)}\label{item:form1} $1,1,2,\underbrace{1,{\underbrace{4,\ldots,4}_{k\text{-times}},3},\ldots{},
    {1,\underbrace{4,\ldots,4}_{k\text{-times}}},3}_{\ell\text{-times}},1,1$
\myitem{(B)}\label{item:form2} $2,\underbrace{1,{\underbrace{4,\ldots,4}_{k\text{-times}},3},\ldots{},
    {1,\underbrace{4,\ldots,4}_{k\text{-times}}},3}_{j\text{-times}}$
\myitem{(C)}\label{item:form3} $2,\underbrace{1,{\underbrace{4,\ldots,4}_{k\text{-times}},3},\ldots{},
    {1,\underbrace{4,\ldots,4}_{k\text{-times}}},3}_{\ell\text{-times}},1,1,2$
\myitem{(D)}\label{item:form4} $\underbrace{1,{\underbrace{4,\ldots,4}_{k\text{-times}},3},\ldots{},
    {1,\underbrace{4,\ldots,4}_{k\text{-times}}},3}_{j\text{-times}},1,1$ \end{enumerate}
for some $j\ge 1$.
\end{lem}

\begin{proof}
We consider possible cases for where the run that sums up to $6+4i$ can start.
\begin{enumerate}
    \item \textbf{Case I}: the run starts at the first 1. Since $1+1+2+1=5<6$, the run must contain $1,1,2,1$.
    \begin{enumerate}
        \item \textbf{Subcase 1}:  If the run ends at $4$, the run sums up to $1+4m$ for some $m\ge 0$.
        \item \textbf{Subcase 2}:  If the run ends at $3$, we have $4m$ for some $m\ge 0$.
        \item \textbf{Subcase 3}:  If the run ends with $3,1$, we have $1+4m$ for some $m\ge 0$.
        \item \textbf{Subcase 4}:  If the run ends with $3,1,1$, we have $6+4m$ for some $m\ge 0$. We have form \ref{item:form1}.
        \item \textbf{Subcase 5}:  If the run ends with $3,1,1,2$, we have $4m$ for some $m\ge 0$.
        \item \textbf{Subcase 6}:  If the run ends with $3,1,1,2,1$, we have $1+4m$ for some $m\ge 0$.
    \end{enumerate}
    \item \textbf{Case II}: the run starts at the second $1$. As above, the run must contain $1,2,1$. Using the same argument,
    we see that there are no such runs that sum up to $6+4i$.
    \item \textbf{Case III}: the run starts at the first 2. As above, the run must contain $2,1$. Using the same argument, we see that to have $6+4i$, the run must either end at $3$ or end with $3,1,1,2$.
    We have form \ref{item:form2} and form \ref{item:form3}.
    \item \textbf{Case IV}: the run starts with $1,4$; it must end with $3,1,1$.
    \item \textbf{Case V}: the run starts at $4$; there are no such runs.
    \item \textbf{Case VI}: the run starts at $3$; there are no such runs.
\end{enumerate}
We have iterated through all possible cases and thus the proof is complete.
\end{proof}

\begin{lem}\label{diffset}
Let $k$ and $\ell\in \mathbb{N}$ be chosen. Consider
\begin{align*}
    S=(0\,|\,1,1,2,
    \underbrace{1,{\underbrace{4\ldots,4}_{k\text{-times}},3},\ldots{},
    {1,\underbrace{4\ldots,4}_{k\text{-times}}},3}_{\ell\text{-times}},1,1,2,1).
\end{align*}
Then the set $[1,\max(S)]\setminus (S-S)$ of missing positive differences is exactly \begin{align*}T\ =\ [6,6+4(k-1)]_4&\cup [14+4(k-1),14+8(k-1)]_4\cup \\ \cdots&\cup [6+8(\ell-1)+4(\ell-1)(k-1), 6+8(\ell-1)+4\ell(k-1)]_4.\end{align*}
\end{lem}
(Recall from Section \ref{not} that $[a,b]_4:=\{x\,|\,x\equiv a\pmod{4}\mbox{ and }a\le x\le b\}$.)
\begin{exa}
We use $S_{2,3}$ as an example. We have \begin{align*}S_{2,3}-S_{2,3}\ =\ [-45,45]\backslash \{\pm 6,\pm 10,\pm 18,\pm 22 ,\pm 30,\pm 34\}.\end{align*}
Note that \begin{align*}\{ 6,10,18,22 , 30, 34\}\ =\ [6,10]_4\cup[18,22]_4\cup[30,34]_4,\end{align*} which is the set $T$ in Lemma \ref{diffset}.
\end{exa}
\begin{proof}
Pick $1\le i\le \ell$. We show that $S-S$ misses
\[
[6+8(i-1)+4(i-1)(k-1), 6+8(i-1)+4i(k-1)]_4;
\]
equivalently, there are no runs that sum up to $6+8(i-1)+4(i-1)(k-1)+4m$ for all $0\le m\le k-1$. We prove this by contradiction.
Pick some $0\le m\le k-1$. Suppose that such a run exists; the run must be one of the forms in Lemma \ref{forms}.
Notice that \begin{align*}6+8(i-1)+4(i-1)(k-1)+4m&\ \le \ 6+8(i-1)+4i(k-1)\\&\ \le\  6+8(\ell-1)+4\ell(k-1).\end{align*}
Since both form \ref{item:form1} and form \ref{item:form3} in Lemma \ref{forms}
gives $6+8\ell+4(k-1)\ell$, our run must be of the form \ref{item:form2} or \ref{item:form4}.
We consider these two cases.
\begin{enumerate}
    \item \textbf{Case I}: the run is of form \ref{item:form2}. Then it sums up to $2+4(k+1)j$ for some $j\ge 1$. We have:
    \begin{equation}
    \begin{split}
    2+4(k+1)j\ &=\ 6+8(i-1)+4(i-1)(k-1)+4m\\
    (k+1)(j-i+1)\ &=\ 1+m.\
    \end{split}
    \end{equation}
    So, $1\le (k+1)(j-i+1)=1+m\le k$ and so, $0< j-i+1< 1$, which is a contradiction.
    \item \textbf{Case II}: the run is of the form \ref{item:form4}. Then it sums up to $4(k+1)j+2$ for some $j\ge 1$. As above, we find a contradiction. We have shown that $S-S$ misses $T$.
\end{enumerate}
To complete the proof, we show that $S-S$ contains $[0,9+4(k+1)\ell]\backslash T$.  Note that close to the beginning of the SCD, we have $1+2+1=4$ and after that, the sequence implicitly contains consecutive differences of $4$ (because $3+1=4$ and $1+2+1=4.)$ So, $S-S$ contains all numbers in $[0,9+4(k+1)\ell]$ that are $0\Mod 4$. Similarly, it is not hard to see that $S-S$ contains all numbers that are either $1\Mod 4$ or $3\Mod 4$. Next, we show that all numbers that are $ 2\Mod 4$ and not in $T$ are in $S-S$. We have $2\in S-S$ and $10+4(k-1)\in S-S$ because $10+4(k-1)=(2+1)+4(k-1)+(4+3)$. Therefore, $\{10+4(k-1)+(1+4k+3)i\,|\,0\le i\le\ell-1
\}\subseteq S-S$. Lastly, $6+4(k+1)\ell \in S-S$.
because we have the run $2,
    \underbrace{1,{\underbrace{4\ldots,4}_{k\text{-times}},3},\ldots{},
    {1,\underbrace{4\ldots,4}_{k\text{-times}}},3}_{\ell\text{-times}},1,1,2$.
\end{proof}

\begin{cor}\label{cordiffset}
Choose $k$ and $\ell\in \mathbb{N}$, and let
\begin{align*}
    S=(0\,|\,1,1,2,
    \underbrace{1,{\underbrace{4\ldots,4}_{k\text{-times}},3},\ldots{},
    {1,\underbrace{4\ldots,4}_{k\text{-times}}},3}_{\ell\text{-times}},1,1,2,1).
\end{align*}
Then $|S-S|=19+\ell(6k+8)$.
\end{cor}

\begin{lem}\label{sumset1}
Choose $k\ge 2$ and $\ell\in \mathbb{N}$, and let
\begin{align*}
    S=(0\,|\,1,1,2,
    \underbrace{1,{\underbrace{4\ldots,4}_{k\text{-times}},3},\ldots{},
    {1,\underbrace{4\ldots,4}_{k\text{-times}}},3}_{\ell\text{-times}},1,1,2,1).
\end{align*}
Then $S+S$ contains $[0,8(k+1)\ell+18]\backslash T$, where
\begin{align*}
T\ =\ &\left(\bigcup_{i\in [1,\ell]}\left[12+12(i-1)+4(i-1)(k-2),12+12(i-1)+4i(k-2)\right]_4\right) \\
&\cup\ \Bigg(\bigcup_{i\in [1,\ell]}\left[12\ell+4\ell(k-2)+16+4(i-1)(k-2)+12(i-1),\right.\\
&\ \qquad \qquad \qquad \left.12\ell+4\ell(k-2)+16+4i(k-2)+12(i-1)\right]_4\Bigg).
\end{align*}
\end{lem}

\begin{proof}
Observe that $S_1=\{i\,|\,1\le i\le 4(k+1)\ell+9\mbox{ and }i\equiv 1\Mod 4\}\subseteq S$. So, all of the following sets are in $S+S$:
\begin{align*}
    0+S_1\ &=\ \{i\,|\,1\le i\le 4(k+1)\ell+9\mbox{ and }i\equiv 1\Mod 4\},\\
    1+S_1\ &=\ \{i\,|\,2\le i\le 4(k+1)\ell+9\mbox{ and }i\equiv 2\Mod 4\},\\
    2+S_1\ &=\ \{i\,|\,3\le i\le 4(k+1)\ell+9\mbox{ and }i\equiv 3\Mod 4\},\\
    4(k+1)\ell+9+S_1\ &=\ \{i\,|\,4(k+1)\ell+10\le i\le 8(k+1)\ell+18\mbox{ and }i\equiv 2\Mod 4\},\\
    4(k+1)\ell+8+S_1\ &=\ \{i\,|\,4(k+1)\ell+9\le i\le 8(k+1)\ell+17\mbox{ and }i\equiv 1\Mod 4\},\\
    4(k+1)\ell+6+S_1\ &=\ \{i\,|\,4(k+1)\ell+7\le i\le 8(k+1)\ell+15\mbox{ and }i\equiv 3\Mod 4\}
\end{align*}
Thus $S+S$ contains all numbers that are either $1,2$ or $3\Mod 4$ in the interval $[0,8(k+1)\ell+18]$. Now, we focus on numbers that are divisible by $4$. Observe that
$S_2 =\{0,4,4+4(k+1),4+8(k+1),\ldots,4+4\ell(k+1),8+4\ell(k+1)\}\subseteq S$. We write
\begin{align*}
    S_2\ = \ \{0,8+4\ell(k+1)\}\cup\{4+4i(k+1)|0\le i\le \ell\}.
\end{align*}
We show that all numbers divisible by $4$ that are not in $T$ are in $S+S$. The set of all numbers divisible by $4$ that are not in $T$ is
\begin{align*}\{0,4,8,12+4\ell(k+1),&8\ell+8k\ell+16\}\cup (4+\{12+12(i-1)+4i(k-2)|1\le i\le \ell\})\\ &\cup (8+\{12+12(i-1)+4i(k-2)|1\le i\le \ell\})\\&\cup (20+\{12\ell+4(\ell+i)(k-2)+12(i-1)|1\le i\le \ell\})\\&\cup (24+\{12\ell+4(\ell+i)(k-2)+12(i-1)|1\le i\le \ell\}).\end{align*}
We know the following:
\begin{enumerate}
\item Because $0,4,8+4(k+1)\ell\in S$, $\{0,4,8, 12+4\ell(k+1),8\ell+8k\ell+16\}\subseteq S+S$.
\item For each $1\le i\le \ell$, we have $4+(12+12(i-1)+4i(k-2))=0+(4+4i(k+1))\in S+S$.
\item For each $1\le i\le \ell$, we have $8+(12+12(i-1)+4i(k-2))=4+(4+4i(k+1))\in S+S$.
\item For each $1\le i\le \ell$, we have
$20+12\ell+4(\ell+i)(k-2)+12(i-1)=(4+4\ell(k+1))+(4+4i(k+1))\in S+S$.
\item For each $1\le i\le \ell$, we have
$24+12\ell+4(\ell+i)(k-2)+12(i-1)=(8+4\ell(k+1))+(4+4i(k+1))\in S+S$.
\end{enumerate}
We have shown that all numbers divisible by $4$ that are not in $T$ are in $S+S$, and this completes the proof.
\end{proof}

\begin{lem}\label{sumset2}
Choose $k\ge 2$ and $\ell\in \mathbb{N}$, and let
\begin{align*}
    S\ = \ (0\,|\,1,1,2,
    \underbrace{1,{\underbrace{4\ldots,4}_{k\text{-times}},3},\ldots{},
    {1,\underbrace{4\ldots,4}_{k\text{-times}}},3}_{\ell\text{-times}},1,1,2,1).
\end{align*}
Then $S+S$ contains none of the elements in
\begin{align*}
T\ =\ &\Bigg(\bigcup_{i\in [1,\ell]}\big[12+12(i-1)+4(i-1)(k-2),12+12(i-1)+4i(k-2)\big]_4\Bigg) \\
&\cup\ \Bigg(\bigcup_{i\in [1,\ell]}\big[12\ell+4\ell(k-2)+16+4(i-1)(k-2)+12(i-1),\\
&\ \qquad \qquad \qquad 12\ell+4\ell(k-2)+16+4i(k-2)+12(i-1)\big]_4\Bigg).
\end{align*}
\end{lem}

\begin{proof}
To complete the proof, we prove that
none of the numbers in $T$ are in $S+S$. We write out $S$ explicitly:
\begin{align*}
    S\ =\ \{0,1,2,4,5\}\ &\cup\ \{5+4(j-1)(k+1)+4i\,|\,1\le j\le \ell,1\le i\le k\}\\
    &\cup \ \{4+4i(k+1)|1\le i\le \ell\}\ \cup \ \{5+4i(k+1)|1\le i\le \ell\}\\
    &\cup\ \{6+4\ell(k+1),8+4\ell(k+1), 9+4\ell(k+1)\}.
\end{align*}
We consider elements in \begin{align*}\bigcup_{i\in [1,\ell]}\big[12+12(i-1)+4(i-1)(k-2),12+12(i-1)+4i(k-2)\big]_4.\end{align*} Pick $1\le m\le \ell$ and $0\le n\le k-2$. Consider
\begin{align*}
12+12(m-1)+4(m-1)(k-2)+4n\
&= \ 4m(k+1)-4k+8+4n \\
&\le \ 4\ell+4k\ell-4k+8+4(k-2)\\
&=\ 4\ell(k+1).
\end{align*}
Because $S$ contains no numbers that are $3\Mod 4$, for a pair whose sum is $4m(k+1)-4k+8+4n$, we cannot use numbers that are $1\Mod 4$. Also, because $4m(k+1)-4k+8+4n\le 4\ell(k+1)$, we can ignore all numbers that are greater than $4\ell(k+1)$. Hence, our set of concern is
\begin{align*}
 \{0,2,4\}\ \cup \ \{4+4i(k+1)\,|\,1\le i\le \ell\}.
\end{align*}
If a pair that sums to $4m(k+1)+4(n-k)+8$ is in $\{4+4i(k+1)\,|\,1\le i\le \ell\}$, then there exists $m'$ and $n'$ such that
\begin{align*}
8+4(m'+n')(k+1)\ &=\ 4m(k+1)+4(n-k)+8,\\
(m'+n'-m+1)(k+1)\ &=\ n+1.
\end{align*}
Thus, $0<(m'+n'-m+1)(k+1)=n+1\le k-1$ and we get, $0<m'+n'-m+1<1$, a contradiction. Therefore, one of the number is in $\{0,2,4\}$. Let $4+4m'(k+1)$ be the number used in $\{4+4i(k+1)|1\le i\le \ell\}$. We consider three cases corresponding to the three elements in $\{0,2,4\}$.
\begin{enumerate}
    \item We have $0+(4+4m'(k+1))=4m(k+1)-4k+8+4n$. So, $0<(k+1)(m'-m+1)=n+2\le k$, which implies $0<m'-m+1<1$, a contradiction.
    \item We have $2+(4+4m'(k+1))=4m(k+1)-4k+8+4n$. So, $0<2(m'-m+1)(k+1)=3+2n\le 2k-1$, which implies $0<m'-m+1<1$, a contradiction.
    \item We have $4+(4+4m'(k+1))=4m(k+1)-4k+8+4n$. So, $0<(m'-m+1)(k+1)=n+1\le k-1$, which implies $0<m'-m+1<1$, a contradiction.
\end{enumerate}
Next, we consider elements in
\[
\bigcup_{i\in [1,\ell]}\big[12\ell+4\ell(k-2)+16+4(i-1)(k-2)+12(i-1)
,12\ell+4\ell(k-2)+16+4i(k-2)+12(i-1)\big]_4.
\]Pick $1\le m\le \ell$ and $0\le n\le k-2$. Consider
\begin{align*}
    12\ell&+4\ell(k-2)+16+4(m-1)(k-2)+12(m-1)+4n\\
    &= \ 4(\ell+m)(k+1)+12-4k+4n\\
    &\ge \ 4(\ell+1)(k+1)+12-4k\ =\ 4(k+1)\ell+16.
\end{align*}
Thus, we cannot use any of $0,2,4$ in our pair.
As above, the set which concerns us is \begin{align*} \{4+4i(k+1)\,|\,1\le i\le \ell\}\ \cup\ \{6+4\ell(k+1),8+4\ell(k+1)\}.
\end{align*}
If a pair that sums to $4(\ell+m)(k+1)+12-4k+4n$ is in $\{4+4i(k+1)|1\le i\le \ell\}$,  then for some $1\le m',n'\le \ell$ we have:
\begin{equation*}
\begin{split}
    4(\ell+m)(k+1)+12-4k+4n\ &=\ 8+4(k+1)(m'+n')\\
    0\ <\ (-\ell-m+m'+n'+1)(k+1)\ &=\ n+2\ \le \ k
\end{split}
\end{equation*}
So, $0<-\ell-m+m'+n'+1<1$, a contradiction. Therefore, a number in the pair must be in $\{6+4\ell(k+1),8+4\ell(k+1)\}$. Since $4(\ell+m)(k+1)+12-4k+4n\le 8\ell(k+1)+4$, both numbers cannot be in $\{6+4\ell(k+1),8+4\ell(k+1)\}$. We consider two cases corresponding to the two elements of $\{6+4\ell(k+1),8+4\ell(k+1)\}$:
\begin{enumerate}
    \item We have $(6+4\ell(k+1))+(4+4m'(k+1))=4(\ell+m)(k+1)+12-4k+4n$. Equivalently, $0<2(k+1)(m'-m+1)=2n+3\le 2(k-2)+3=2k-1$ and so $0<m'-m+1<1$, a contradiction.
    \item We have $(8+4\ell(k+1))+(4+4m'(k+1))=4(\ell+m)(k+1)+12-4k+4n$. Equivalently, $0<(k+1)(m'-m+1)=n+1\le (k-2)+1=k-1$ and so, $0<m'-m+1<1$, a contradiction.
\end{enumerate}
This completes our proof.
\end{proof}
\begin{exa}
We use $S_{2,3}$ as an example to illustrate Lemma \ref{sumset1} and Lemma \ref{sumset2}. We have \begin{align*}S_{2,3}+S_{2,3}\ =\ [0,90]\backslash \{12,24,36,52,64,76\},\end{align*}
and the set $\{12,24,36,52,64,76\}$ is exactly the set $T$ in Lemmas \ref{sumset1} and Lemma \ref{sumset2}.
\end{exa}
\begin{cor}\label{corsumset}
Let $k\ge 2$ and $\ell\in \mathbb{N}$ be chosen. Let
\begin{align*}
    S\ = \ (0\,|\,1,1,2,
    \underbrace{1,{\underbrace{4\ldots,4}_{k\text{-times}},3},\ldots{},
    {1,\underbrace{4\ldots,4}_{k\text{-times}}},3}_{\ell\text{-times}},1,1,2,1).
\end{align*}
Then $|S+S|=19+\ell(6k+10)$.
\end{cor}
%\begin{cor}
%Let $k\ge 2$ and $\ell\in \mathbb{N}$ be chosen. Let
%\begin{align*}
%    S=(0\,|\,1,1,2,
%    \underbrace{1,{\underbrace{4\ldots,4}_{k\text{-times}},3},\ldots{},
%    {1,\underbrace{4\ldots,4}_{k\text{-times}}},3}_{\ell\text{-times}},1,1,2,1).
%\end{align*}
%Then $|S+S|-|S-S|=2\ell$.
%\end{cor}

\begin{proof}[Proof of Theorem \ref{periodic}, Item \ref{firsti}] The proof follows immediately from Corollaries \ref{cordiffset} and \ref{corsumset}, because
\[
|S_{k,\ell}+S_{k,\ell}| - |S_{k,\ell}-S_{k,\ell}|
= [19+\ell(6k+10)] - [19+\ell(6k+8)] = 2\ell.\qedhere
\]
\end{proof}

\begin{rek}
Theorem \ref{periodic} is a generalization of Theorem 2 and Theorem 3 in \cite{PW}. With a little more work, we can show that sets in $\mathcal{F}_{per}$ are RSD for $\ell$ sufficiently large.
\end{rek}

We also offer another family of MSTD sets formed by repeating certain interior blocks. We do not prove the theorem since it is not in the focus of the current paper. However, the proof is very similar to the proof of Theorem \ref{periodic} but replacing ``modulo $4$'' by ``modulo $k$'' throughout.

\begin{thm}\label{anotherdim}
For $k\ge 4$, the following is a MSTD set:
$$A_{k,1}\ = \ (0\,|\,\underbrace{1,1,\ldots,1}_{k-2\text{-times}},2,1,k\underbrace{k+1,k+1,\ldots,k+1}_{k-4\text{-times}},3,\underbrace{1,1,\ldots,1}_{k-2\text{-times}},2,1),$$
and $|A_{k,1}+A_{k,1}|-|A_{k,1}-A_{k,1}|=2$. Define $A_{k,\ell}$ to be a similarly built set with the sequence $1,k,\underbrace{k+1,k+1,\ldots,k+1}_{k-4\text{-times}},3$ repeated $k$ times, then
\[
|A_{k,\ell}+A_{k,\ell}|-|A_{k,\ell}-A_{k,\ell}|
=2\ell
\]
\end{thm}

\begin{rek}
If we consider Theorem \ref{periodic} to be a generalization of Theorem 2 and Theorem 4 in \cite{PW}, then Theorem \ref{anotherdim} is another generalization from a different perspective. Notice that $A_{4,\ell}=S_{1,\ell}$. Sets $A_{k,\ell}$ also form a family of MSTD sets and RSD sets with interior blocks.
\end{rek}

%%%%%%%%%%%%%%%%%%%%%%%%%%%%%%%%%%%%%%%%%%%%%%%%%%%%%%%%%%%%%%%%%%%%%%%%%%%%%%%%%%%%%%%%%%%%%%%%%%%%%%%%%%%%%%%%%%%%%%%%%%%%%%%%%%%%%%%
%%%%%%%%%%%%%%%%%%%%%%%%%%%%%%%%%%%%%%%%%%%%%%%%%%%%%%%%%%%%%%%%%%%%%%%%%%%%%%%%%%%%%%%%%%%%%%%%%%%%%%%%%%%%%%%%%%%%%%%%%%%%%%%%%%%%%%%
%%%%%%%%%%%%%%%%%%%%%%%%%%%%%%%%%%%%%%%%%%%%%%%%%%%%%%%%%%%%%%%%%%%%%%%%%%%%%%%%%%%%%%%%%%%%%%%%%%%%%%%%%%%%%%%%%%%%%%%%%%%%%%%%%%%%%%%

\section{Good Properties of the Family $\mathcal{F}$}

%%%%%%%%%%%%%%%%%%%%%%%%%%%%%%%%%%%%%%%%%%%%%%%%%%%%%%%%%%%%%%%%%%
%%%%%%%%%%%%%%%%%%%%%%%%%%%%%%%%%%%%%%%%%%%%%%%%%%%%%%%%%%%%%%%%%%
%%%%%%%%%%%%%%%%%%%%%%%%%%%%%%%%%%%%%%%%%%%%%%%%%%%%%%%%%%%%%%%%%%
\subsection{Sets $A$ with Large $\mathbf{\log|A+A|/\log|A-A|}$}
The first application of our family $\mathcal{F}$ is that the family produces many sets $A$ with large value of $\log |A+A|/\log |A-A|$. For convenience, we define $f(A):=\log |A+A|/\log |A-A|$. An early example of a set $A$ with high $f(A)$ is given by Hegarty \cite{He}. The set is \begin{align*}A_{15}\ =\ \{0, 1, 2, 4, 5, 9, 12, 13, 17, 20, 21, 22, 24, 25, 29, 32, 33, 37, 40, 41, 42, 44, 45\}.\end{align*}
In our notation, $$A_{15}\ =\ (0\,|\,1,1,2,1,4,3,1,4,3,1,1,2,1,4,3,1,4,3,1,1,2,1),$$ which is very close to a set in our family $\mathcal{F}_{per}$, namely $$S_{1,4}\ =\ (0\,|\,1,1,2,1,4,3,1,4,3,1,4,3,1,4,3,1,1,2,1).$$ It turns out that $f(S_{1,4})>f(A_{15})$.

When analyzing the periodic subfamily $\mathcal{F}_{per}$ defined in Theorem \ref{periodic}, we find the set $S_{1,6}$ with the property of $f(S_{1,6})=1.023777\ldots$, which is larger than previous results in \cite{He} (1.0208\ldots) and \cite{AMMS} (1.0213\ldots) but smaller than the current record in \cite{PW} (1.03059\ldots). It is worth noting, however, the set with the highest known value of $f(A)$ ($1.3059\ldots)$ exhibited by Penman and Wells \cite{PW} is a member of the family $\mathcal{F}$:
$$(0\,|\,1,1,2,1,4,3,1,4,4,3,
    \underbrace{1,{\underbrace{4,\ldots,4}_{3\text{-times}},3},\ldots{},
    {1,\underbrace{4,\ldots,4}_{3\text{-times}}},3}_{9\text{-times}},1,4,4,3,1,4,3,1,1,2,1).$$
In fact, there are at least 22 sets $A$ in $\mathcal{F}$ with $f(A)>1.03$: these sets are of the form:
\begin{align*}
    &(0\,|\,1,1,2,1,4,3,1,4,4,3,
    \underbrace{1,{\underbrace{4,\ldots,4}_{3\text{-times}},3},\ldots{},
    {1,\underbrace{4,\ldots,4}_{3\text{-times}}},3}_{\ell\text{-times}},1,4,4,3,1,4,3,1,1,2,1),\\
    &(0\,|\,1,1,2,1,4,3,1,4,4,3,
    \underbrace{1,{\underbrace{4,\ldots,4}_{3\text{-times}},3},\ldots{},
    {1,\underbrace{4,\ldots,4}_{3\text{-times}}},3}_{\ell\text{-times}},1,4,4,3,1,4,3,1,1,2).
\end{align*}

%%%%%%%%%%%%%%%%%%%%%%%%%%%%%%%%%%%%%%%%%%%%%%%%%%%%%%%%%%%%%%%%%%
%%%%%%%%%%%%%%%%%%%%%%%%%%%%%%%%%%%%%%%%%%%%%%%%%%%%%%%%%%%%%%%%%%
%%%%%%%%%%%%%%%%%%%%%%%%%%%%%%%%%%%%%%%%%%%%%%%%%%%%%%%%%%%%%%%%%%
\subsection{Economical Way to Construct a Set $A$ with Fixed $|A+A|-|A-A|$}
We show another application of our large family $\mathcal{F}$ of MSTD sets, which is to construct sets with a fixed difference $|A+A|-|A-A|$ economically, i.e. with a relatively small width between their maximum and minimum elements.

\begin{proof}[Proof of Theorem~\ref{econ}]
Fix $x\in\mathbb{N}$. If $x$ is even, pick $k=1$ and $\ell=x/2$ for sets in Item \ref{firsti} of Theorem \ref{periodic}; then by Theorem we find a set $A=S_{k,\ell}$ with
\[
|S_{k,\ell}+S_{k,\ell}|-|S_{k,\ell}-S_{k,\ell}|=2\ell = x
\]
and $\max A=9+8\ell=9+4x$, $\min A=0$.

If $x$ is odd, pick $k=1$ and $\ell=(x+1)/2$ for sets in Item \ref{secondi} of Theorem \ref{periodic}, we find a set $A=S'_{k,\ell}$ with
\[
|S'_{k,\ell}+S'_{k,\ell}|-|S'_{k,\ell}-S'_{k,\ell}|=2\ell-1=x
\]
and $\max A=8+8\ell=12+4x$, $\min A=0$.

Hence, for any positive integer $x$, there exists $A\subseteq [0,12+4x]$ such that $|A+A|-|A-A|=x$.
\end{proof}

\begin{rek}
Linear growth of the interval containing $A$ is the best we can do. To see this, assume that the theorem is true for
$A_x\subseteq [0,\phi(x)]$, where $\phi(x)$ is sub-linear. We have:
\begin{equation}
\lim_{x\rightarrow\infty}\frac{|A_x+A_x|-|A_x-A_x|}{\phi(x)}\ =\ \lim_{x\rightarrow\infty}\frac{x}{\phi(x)}\ =\ \infty,
\end{equation}
which is a contradiction, since
for all sets $A_x\subseteq [0,\phi(x)]$ for large enough $x$,
\[
\frac{|A_x+A_x|-|A_x-A_x|}{\phi(x)}\le\frac{|A_x+A_x|}{\phi(x)}<\frac{2\phi(x)+1}{\phi(x)}\le 3.
\]
\end{rek}

%%%%%%%%%%%%%%%%%%%%%%%%%%%%%%%%%%%%%%%%%%%%%%%%%%%%%%%%%%%%%%%%%%
%%%%%%%%%%%%%%%%%%%%%%%%%%%%%%%%%%%%%%%%%%%%%%%%%%%%%%%%%%%%%%%%%%
%%%%%%%%%%%%%%%%%%%%%%%%%%%%%%%%%%%%%%%%%%%%%%%%%%%%%%%%%%%%%%%%%%
\subsection{Small Fringe Size Generator -- Proof of Theorem \ref{boundsforRSD}} Many classes of MSTD sets can be generated by finding a good fringe pair, i.e.\ the two sets of elements on the leftmost and rightmost sides of the interval $[0,n]$. Examples can be found the proofs of Theorem 8 in \cite{He}, Theorem 1.4 in \cite{AMMS}, Theorem 1 in \cite{MO}, Theorem 1.1 in \cite{MOS} and Theorem 17 in \cite{PW}. Often, when shifted close to each other, two sets in a fringe pair form an MSTD set. However, these fringe pairs have been found by brute force and there has not been a systematic way to generate fringes. It turns out that $\mathcal{F}$ can be a good fringe generator; we demonstrate this by improving the lower bound for the proportion of RSD sets of $\{0,1,\dots,n-1\}$ mentioned in Theorem 17 \cite{He}.

In particular, Pennman and Wells used a fringe pair of size 120 generated by the fringe pair used in \cite{MO}. The method is to repeat blocks of sets, which inefficiently creates a small lower bound of about $10^{-37}$. The authors mentioned that Zhao's techniques can be modified to improve the result; however, this task requires a substantial computation. We believe that this is true since RSD sets are much rarer than MSTD sets\footnote{Exhaustive computer search shows that there are no RSD subsets of $[0,29]$, while there are at least $4.5\cdot 10^{5}$ MSTD sets in the same interval.}. As Zhao's technique relies on extensive search for fringe pairs, the technique is much less effective when applied to RSD sets. Therefore, a feasible and simple way to improve the bound is to find a better fringe pair. Here is a fringe pair generated by $\mathcal{F}$ (we use $L$ and $U$ to match the notations with \cite{PW}):
\begin{align*}
L\ &= (0\,|\,1,1,2,1,4,3,1,4,3,1,4,3,1,4,3,1,1,1)\\
&= \{0,1,2,4,5,9,12,13,17,20,21,25,28,29,33,36,37,38,39\},\\
U\ &=\ (n-41)+(0\,|\,1,1,1,1,4,3,1,4,3,1,4,3,1,4,3,1,1,2,1)\\
&=\ (n-41)+\{0, 1, 2, 3, 4, 8, 11, 12, 16, 19, 20, 24, 27, 28, 32, 35, 36, 37, 39, 40\}\\
&= \ n-\{41, 40, 39, 38, 37, 33, 30, 29, 25, 22, 21,17, 14, 13, 9, 6, 5, 4, 2, 1\}.
\end{align*}
Observe that the fringe pair is formed by the MSTD set $S_{1,5}$. We have:
\begin{align*}
    L\hat{+}L\ &=\ [0,78]\backslash \{0,8,78\},\\
    U\hat{+}U\ &=\ [n-41,n+38],\\
    U\hat{+}U\ &=\ [2n-82,2n-2]\backslash\{2n-2,2n-4,2n-12,2n-82\}.
\end{align*}
Notice that $U-L$ misses $\pm (n-12), \pm (n-20), \pm (n-28), \pm (n-36)$. Hence, $S-S$ misses at least 8 numbers.
If we can guarantee that $S\hat{+}S$ misses only $7$ elements in $\{0,8,78,\pm (2n-82), \pm (2n-12), \pm (2n-4),
\pm (2n-2)\}$, then $S$ is RSD. Following the proof of Theorem 17 in \cite{PW}, we find a lower bound of
\begin{equation}
(1-8(2^{-19}+2^{-20}))\cdot 2^{-(40+41)}=4.135\cdot 10^{-25}.
\end{equation}
This improvement comes from the reduction in fringe size from $120$ to $81$. Can we find a better bound for the proportion of RSD subsets of $\{0,1,2,\ldots,n-1\}$ as $n\rightarrow\infty?$ Since there are no RSD subsets in $[0,29]$, if we look for a better fringe pair, which is built from an RSD set, the fringe must be of size at least 31. Then the best lower bound that can be achieved by this method is about $2^{-31}\approx 10^{-10}$.

%%%%%%%%%%%%%%%%%%%%%%%%%%%%%%%%%%%%%%%%%%%%%%%%%%%%%%%%%%%%%%%%%%%%%%%%%%%%%%%%%%%%%%%%%%%%%%%%%%%%%%%%%%%%%%%%%%%%%%%%%%%%%%%%%%%%%%%
%%%%%%%%%%%%%%%%%%%%%%%%%%%%%%%%%%%%%%%%%%%%%%%%%%%%%%%%%%%%%%%%%%%%%%%%%%%%%%%%%%%%%%%%%%%%%%%%%%%%%%%%%%%%%%%%%%%%%%%%%%%%%%%%%%%%%%%
%%%%%%%%%%%%%%%%%%%%%%%%%%%%%%%%%%%%%%%%%%%%%%%%%%%%%%%%%%%%%%%%%%%%%%%%%%%%%%%%%%%%%%%%%%%%%%%%%%%%%%%%%%%%%%%%%%%%%%%%%%%%%%%%%%%%%%%

\section{Observation: Interior Block Sizes and The Growth of $|A+A|-|A-A|$}
Spohn \cite{Sp} was the first to share the concept of and raise several questions about interior blocks existing within MSTD sets. He noted that the repetition of interior blocks may increase the cardinality of the sum set by more than that of the difference set. For a set $A$ having an interior block $B_A$, let $T_A$ be the value that the sum set increase by more than the difference set when $B_A$ is repeated. We observe the relationship between $T_A$ and $|B_A|$.
Theorem \ref{periodic} gives us the following:

\begin{thm}
The following results about $T_A/|B_A|$ are true.
\begin{enumerate}
\item \label{growth1} There exists a set $A$ such that $T_A/|B_A|=0$.
\item \label{growth2} For any $\varepsilon>0$, there exists a set $A$ with $0<T_A/|B_A|<\varepsilon$.
\item \label{growth3}For any $0\le \varepsilon<0.2$, there exists a set $A$ such that $T_A/|B_A|>1+\varepsilon$.
\end{enumerate}
\end{thm}

\begin{proof}\
\begin{enumerate}

\item Consider the set $S_{1,1}$ in Theorem \ref{periodic}. Notice that $|S_{2,1}+S_{2,1}|-|S_{2,1}-S_{2,1}|=|S_{1,1}+S_{1,1}|-|S_{1,1}-S_{1,1}|=2$ and so, $T_{S_{1,1}}=0$. In other words, repeating $4$ does not change $|S_{1,1}+S_{1,1}|-|S_{1,1}-S_{1,1}|$. This proves (\ref{growth1}).

\item Pick $\varepsilon>0$ and choose $k$ such that $2/(k+2)<\varepsilon$. Consider the set $S_{k,1}$ in Theorem \ref{periodic}. We have $(|S_{k,2}+S_{k,2}|-|S_{k,2}-S_{k,2}|)-(|S_{k,1}+S_{k,1}|-|S_{k,1}+S_{k,1}|)=2$ and so, $T_{S_{k,1}}=2$. Hence, $0<T_{S_{k,1}}/|B_{S_{k,1}}|=2/(k+2)<\varepsilon$. This proves (\ref{growth2}).

\item Finally, Theorem 12 in \cite{PW} shows that $(|Q_{3}+Q_{3}|-|Q_{3}-Q_{3}|)-(|Q_{2}+Q_{2}|-|Q_{2}-Q_{2}|)=6$, while the interior block is 1,4,4,4,3 (of size $5.)$ Hence, $T_{Q_2}/|B_{Q_2}|=6/5=1.2$. This proves (\ref{growth3}). \qedhere
\end{enumerate}
\end{proof}

We care about the relationship between the interior block size and the growth of the sum set with respect to the difference set, because this relationship seems to be closely related to sets $A$ with large $f(A)$. The previous record $A_{15}$ in \cite{He} has $T_{A_{15}}/|B_{A_{15}}|=2/3$, the highest known at that time. The new record $Q_{10}$ in \cite{PW} has $T_{Q_{10}}/|B_{Q_{10}}|=6/5$, which is much higher and this explains why the current record $(1.03059\ldots)$ is much higher than the old record of $(1.0208\ldots)$.

%%%%%%%%%%%%%%%%%%%%%%%%%%%%%%%%%%%%%%%%%%%%%%%%%%%%%%%%%%%%%%%%%%%%%%%%%%%%%%%%%%%%%%%%%%%%%%%%%%%%%%%%%%%%%%%%%%%%%%%%%%%%%%%%%%%%%%%
%%%%%%%%%%%%%%%%%%%%%%%%%%%%%%%%%%%%%%%%%%%%%%%%%%%%%%%%%%%%%%%%%%%%%%%%%%%%%%%%%%%%%%%%%%%%%%%%%%%%%%%%%%%%%%%%%%%%%%%%%%%%%%%%%%%%%%%
%%%%%%%%%%%%%%%%%%%%%%%%%%%%%%%%%%%%%%%%%%%%%%%%%%%%%%%%%%%%%%%%%%%%%%%%%%%%%%%%%%%%%%%%%%%%%%%%%%%%%%%%%%%%%%%%%%%%%%%%%%%%%%%%%%%%%%%

\section{Smallest Cardinality for RSD sets}
Hegarty proved that the smallest MSTD sets have size 8, and there is exactly one such set up to affine transformation. The method is to reduce the problem to finite computations and run through all possible cases by computers. As commented in \cite{He}, this method is not feasible in finding all possible MSTD sets of cardinality 9 since there are many pair of possible equal differences for a set of $9$ random numbers. However, Penman and Wells \cite{He} proved that the list of size-9 MSTD sets given by \cite{He} is exhaustive (up to affine transformation). They also observed that the smallest cardinality of RSD sets must be in the interval $[10,16]$. We narrow this range of possible values for the size of the smallest RSD sets.

\begin{thm}
The smallest cardinality of RSD sets is in the interval $[10,15]$. Furthermore there are no RSD subsets of $[0,29]$ and the smallest diameter of an RSD set is 30.
\end{thm}

There are exactly 6 RSD sets in $[0,30]$, and they all have cardinality 15:
\begin{align*}
    C_1\ &= \ \{0, 1, 2, 3, 6, 8, 13, 16, 18, 23, 24, 26, 28, 29, 30\},\\
    C_2\ &= \ \{0, 1, 2, 3, 6, 9, 14, 15, 17, 22, 23, 26, 28, 29, 30\},\\
    C_3\ &= \ \{0, 1, 2, 4, 5, 8, 9, 14, 18, 21, 22, 26, 27, 28, 30\},\\
     C_4\ &= \ 30-C_1,\\
    C_5\ &= \ 30-C_2,\\
    C_6\ &= \  30-C_3.
\end{align*}
For all $1\le i\le 6$, $|C_i+C_i|-|C_i-C_i|=1$.

%%%%%%%%%%%%%%%%%%%%%%%%%%%%%%%%%%%%%%%%%%%%%%%%%%%%%%%%%%%%%%%%%%%%%%%%%%%%%%%%%%%%%%%%%%%%%%%%%%%%%%%%%%%%%%%%%%%%%%%%%%%%%%%%%%%%%%%
%%%%%%%%%%%%%%%%%%%%%%%%%%%%%%%%%%%%%%%%%%%%%%%%%%%%%%%%%%%%%%%%%%%%%%%%%%%%%%%%%%%%%%%%%%%%%%%%%%%%%%%%%%%%%%%%%%%%%%%%%%%%%%%%%%%%%%%
%%%%%%%%%%%%%%%%%%%%%%%%%%%%%%%%%%%%%%%%%%%%%%%%%%%%%%%%%%%%%%%%%%%%%%%%%%%%%%%%%%%%%%%%%%%%%%%%%%%%%%%%%%%%%%%%%%%%%%%%%%%%%%%%%%%%%%%

\section{Open Questions}
We end with these open questions:
\begin{enumerate}
    \item What are the possible values of $T_A/|B_A|$ over all sets $A$? Is there a set $A$ such that $T_A/|B_A|>1.2?$ This may lead to an increase in the highest known value of $\log|A+A|/\log|A-A|$.
    \item Does $T_A/|B_A|>0$ imply that $|B_A|\ge 3?$
    \item Is Conjecture \ref{conjec} correct?
    \item Can we formalize the concept of interior blocks? When do interior blocks exist?
    \item Can we find a better bound for the proportion of RSD subsets of $[0,n-1]$ as $n\rightarrow\infty?$
    \item What is the size of the smallest RSD sets? Is there a better way to find out this number than Hegarty's method (which requires large computing power)?
\end{enumerate}
\bigskip

\noindent \textbf{Acknowledgement.} The authors were supported by NSF grants DMS1659037 and DMS1561945, the Finnerty Fund, Washington and Lee University and Williams College. We thank the participants from the 2018 SMALL REU program for many helpful conversations.

%%%%%%%%%%%%%%%%%%%%%%%%%%%%%%%%%%%%%%%%%%%%%%%%%%%%%%%%%%%%%%%%%%%%%%%%%%%%%%%%%%%%%%%%%%%%%%%%%%%%%%%%%%%%%%%%%%%%%%%%%%%%%%%%%%%%
%%%%%%%%%%%%%%%%%%%%%%%%%%%%%%%%%%%%%%%%%%%%%%%%%%%%%%%%%%%%%%%%%%%%%%%%%%%%%%%%%%%%%%%%%%%%%%%%%%%%%%%%%%%%%%%%%%%%%%%%%%%%%%%%%%%%
%%%%%%%%%%%%%%%%%%%%%%%%%%%%%%%%%%%%%%%%%%%%%%%%%%%%%%%%%%%%%%%%%%%%%%%%%%%%%%%%%%%%%%%%%%%%%%%%%%%%%%%%%%%%%%%%%%%%%%%%%%%%%%%%%%%%

\end{document}